\documentclass[12pt,letterpaper]{article}
\usepackage{amsmath}
\usepackage{enumerate}
\usepackage{natbib}
\newcommand{\blind}{0}

\usepackage{makeidx}
\usepackage{amsmath,amssymb,amsthm}
\usepackage{graphicx}
\usepackage{color}
\usepackage{verbatim}
\usepackage{longtable}
\usepackage{lscape}
\usepackage{bm}
\usepackage[width=\textwidth]{caption}

\definecolor{blue}{rgb}{0,0,0.9}
\definecolor{red}{rgb}{0.9,0,0}
\definecolor{green}{rgb}{0,0.9,0}
\definecolor{gray}{gray}{0.6}

\newtheorem{proposition}{Proposition}

\newtheorem{theorem}{Theorem}

\def\lam{\lambda} \def\alp{\alpha}
\def\sig{\sigma}
\def\inprod#1#2{\langle#1,#2\rangle}

\newcommand{\R}{\mathbb R}

\newcommand{\primobj}{\text{obj}_\text{primal}}
\newcommand{\dualobj}{\text{obj}_\text{dual}}

\def\diag#1{{\rm diag}(#1)}
\def\norm#1{\|#1\|}

\def\cT{{\cal T}}
\def\bfone{{\bf 1}}

\begin{document}

\def\spacingset#1{\renewcommand{\baselinestretch}%
{#1}\small\normalsize} \spacingset{1.45}


\if0\blind
{
  \title{\bf Fast algorithms for large scale generalized distance weighted discrimination}
	\author{
		Xin Yee Lam\thanks{Department of Mathematics, National
			University of Singapore,
			10 Lower Kent Ridge Road, Singapore
			119076.}, \;
		J.S. Marron\thanks{Department of Statistics and Operations Research,
			University of North Carolina, Chapel Hill, NC 27599-3260 ({\tt marron@email.unc.edu}).},\;
		Defeng Sun\thanks{Department of Mathematics and Risk Management Institute, National
			University of Singapore,
			10 Lower Kent Ridge Road, Singapore
			119076 ({\tt matsundf@nus.edu.sg}).},
		and
		Kim-Chuan Toh\thanks{Department of Mathematics and Institute of Operations Research and Analytics, National
			University of Singapore,
			10 Lower Kent Ridge Road, Singapore
			119076 ({\tt mattohkc@nus.edu.sg}). }
	}
  \maketitle
} \fi

\if1\blind
{
  \bigskip
  \bigskip
  \bigskip
  \begin{center}
    {\LARGE\bf Fast algorithms for large scale generalized distance weighted discrimination}
\end{center}
  \medskip
} \fi

\bigskip
\begin{abstract}
High dimension low sample size statistical analysis is
important in a wide range of applications. In such situations,
the highly appealing discrimination method, support vector machine, can be improved to alleviate data piling at the margin. This leads naturally to the development of distance weighted discrimination (DWD), which can be modeled as a second-order cone programming problem and solved by interior-point methods when the scale  (in sample size and feature dimension)
of the data is moderate. Here, we design a scalable and robust algorithm for solving
large scale generalized DWD problems. Numerical experiments on
real data sets from the UCI repository demonstrate that our algorithm is highly efficient
in solving large scale problems, and sometimes even more efficient
than the highly optimized 
LIBLINEAR and 
LIBSVM for solving the corresponding SVM  problems.
\end{abstract}

\noindent%
\vfill

\section{Introduction}

We consider the problem of finding a (kernelized) linear classifier
for a training data set $\{ (x_i, y_i)\}_{i=1}^n$ with
$x_i\in\R^d$ and the class label $y_i\in\{-1,1\}$ for all $i=1,\ldots,n.$
Here $n$ is the sample size (the number of data vectors available) and
$d$ is the feature dimension.
By far, the most popular and successful method for getting a good linear
classifier from the
training data is
the support vector machine (SVM), originally proposed by \cite{Vapnik}.
Indeed, it has been demonstrated in \cite{FCBA} that the kernel SVM is one of the best
performers in the pool of 179 commonly used classifiers.
Despite its success, it has been observed
in \cite{MarronTodd-2007} that SVM may suffer
from a ``data-piling" problem in the high-dimension-low-sample size (HDLSS) setting
(where the sample size $n$ is smaller than the
feature dimension $d$). The authors proposed a new linear classifier, called ``Distance Weighted Discrimination" (DWD),
as a superior alternative to the SVM.
DWD has become a workhorse method for a number of statistical tasks, including data visualization (\citealt{MarronAlonso}), hypothesis testing linked with visualization in very high dimensions (\citealt{wei2015}), and adjustment for data biases and heterogeneity (\citealt{benito,Liu}).

It is well known that there is a strong need for efficient HDLSS methods
for the settings where $d$ is large, say in the order of $10^4$--$10^5$, especially
in the area of genetic molecular measurements (usually having a small sample size, where many gene level features have been measured), chemometrics (typically a small sample of high dimensional spectra) and medical image analysis (a small sample of 3-d shapes represented by high-dimensional vectors). On the other hand,
given the advent of a huge volume of data
collectible through various sources, especially from the internet,
it is also important for us to consider the case where the sample size $n$ is large, while
the feature dimension may be moderate.
Thus in this paper, we are interested in
the problem of finding a linear classifier
through DWD for
data instances
where $n$ and/or $d$ are large.

In \cite{MarronTodd-2007}, DWD is formulated as a second-order cone
programming (SOCP) problem, and the resulting model is
solved by using a primal-dual interior-point method for SOCP problems implemented
in the software SDPT3 (\citealt{SDPT3}).
However, the IPM based solver employed for DWD in \cite{MarronTodd-2007}
is computationally very expensive for solving problems
where $n$ or $d$ is large, thus making it impractical for large scale problems.
A recent approach to overcome such a computational bottleneck has appeared
in \cite{WangZou}, where the authors proposed  a novel reformulation
of the primal DWD model which consists of minimizing a highly nonlinear convex
objective function subject to a ball constraint. The resulting problem
is solved via its dual and an MM (minimization-majorization) algorithm
is designed to compute the Lagrangian dual function for each given dual multiplier.
The algorithm appears to be  quite promising in theory for solving large scale DWD problems.
However, the current numerical experiments and implementation of the proposed
algorithm in \cite{WangZou} are  preliminary and limited
to small scale data instances, and it appears that substantial work must be
done to make it efficient for large scale instances.
Hence it is premature  to compare our
proposed algorithm with the one in \cite{WangZou}.

The main contribution of this paper is to design
a new method for solving  large scale DWD problems, where we target to solve a problem
with the
sample size $n \approx 10^4$--$10^6$ and/or the dimension $d\approx 10^4$--$10^5$.
Our method is a convergent 3-block semi-proximal alternating direction method of multipliers
(ADMM), which is designed based on the recent advances in research on convergent multi-block
ADMM-type methods (\citealt{STY,LST,CST}) for solving convex composite quadratic conic programming
problems.

The classical ADMM is initially proposed for solving a 2-block convex optimization problem
with a collection of coupling linear constraints. Over the years, there have been many variations of ADMM proposed and applied to a great variety of optimization problems. A natural modification is to extend the original ADMM from two-block to multi-block settings. However, in \cite{directADMM}, it was shown that the directly extended ADMM may not be convergent.
Thus it is necessary to make some modifications to the directly extended ADMM in order to
get  a convergent algorithm. In \cite{STY}, the authors proposed a semi-proximal ADMM for solving a convex conic programming problem with 3 blocks of variables and 4 types of constraints. The algorithm  is a convergent modification of the ADMM with an additional inexpensive step in each iteration. 
 In \cite{LST}, the authors proposed a Schur complement based (SCB) convergent semi-proximal ADMM for solving a multi-block linearly constrained convex programming problem whose objective function is the sum of two proper closed convex functions plus an arbitrary number of convex quadratic or linear functions. One of the key contributions in \cite{LST}  is the discovery of the Schur complement based decomposition method which allows the multi-block subproblems to be solved
efficiently  while ensuring the convergence of the algorithm.
More recently, \cite{LST2} generalized the SCB decomposition
method in \cite{LST} to the {\em inexact} symmetric Gauss-Seidel decomposition method, which
provides an elegant framework and simpler derivation.
Based on this previous research, in \cite{CST}, the authors proposed an {\em inexact} symmetric Gauss-Seidel based multi-block semi-proximal ADMM for solving a class of high-dimensional convex composite conic optimization problems, which has been demonstrated to have much better performance than the possibly non-convergent directly extended ADMM in solving high dimensional convex quadratic semidefinite programming problems.

Inspired by the above works, we propose a convergent 3-block semi-proximal ADMM, which is a
modification of the inexact sGS-ADMM algorithm designed in \cite{CST}
to solve the DWD model.
The first contribution we make is in reformulating the primal formulation of the {\em generalized} DWD model (using the
terminology from \cite{WangZou}) and adapting the
powerful inexact sGS-ADMM framework for solving the  reformulated problem.
This
is in contrast to numerous SVM algorithms which are primarily designed for solving
the dual formulation of the SVM model.
The second contribution we make is in designing highly efficient techniques to
solve the subproblems in each  of the inexact sGS-ADMM iterations.
If $n$ or $d$ is moderate, then the complexity at each iteration is $O(nd)+O(n^2)$ or $O(nd)+O(d^2)$ respectively. If both $n$ and $d$ are large, then we employ the  conjugate gradient iterative method for solving the large linear systems of equations involved.
We also devise various strategies to speed up the practical performance of
the sGS-ADMM algorithm in solving large scale instances (with the largest instance
having $n=256,000$ and $d\approx 3\times 10^6$) of DWD problems
with real data sets from the UCI machine learning repository (\citealt{UCI}).
We should emphasize that the key in achieving high efficiency in our algorithm depends very much on the intricate numerical techniques and
sophisticated implementation we have developed.

Finally,
we conduct extensive numerical experiments to evaluate the performance
of our proposed algorithm against a few other alternatives.
Relative to  the primal-dual interior-point method used in \cite{MarronTodd-2007},
our algorithm is vastly superior in terms of computational time and memory usage
in solving large scale problems, where our algorithm can be a few thousands times
faster. By exploiting all the highly efficient numerical techniques we have developed in
the implementation of the sGS-ADMM algorithm
for solving the generalized DWD problem, we can also
get an efficient implementation of
the possibly non-convergent directly extended ADMM for solving the same problem.
On the tested problems, our algorithm generally requires fewer iterations
compared to the directly extended ADMM even when the latter is convergent.
On quite a few instances, the directly extended ADMM actually requires many more
iterations than our proposed algorithm to solve the problems.
We also compare the efficiency of our algorithm in solving the generalized
DWD problem against the highly optimized {LIBLINEAR} (\citealt{LIBLINEAR}) and LIBSVM (\citealt{LIBSVM}) in solving the
corresponding dual SVM problem. Surprisingly, our algorithm can even be
more efficient than LIBSVM in solving large scale problems even though the
DWD model is more complex, and on
some instances, our algorithm is 50--100 times faster.
{Our DWD model is also able to produce the best test (or generalization) errors 
compared to LIBLINEAR and LIBSVM among the tested instances.}

The remaining parts of this paper are organized as follows. In section \ref{sec-GeneralDWD}, we present the DWD formulation in full detail. In section \ref{sec-sGSADMM}, we propose our inexact sGS-based ADMM method for solving large scale DWD problems. We also discuss some essential computational techniques used in our implementation.
We report our numerical results in section \ref{sec-experiments}. We will also compare the performance of our algorithm to other solvers on the same data sets in this particular section. Finally, we conclude the paper in section \ref{sec-conclusion}.

Notation. We denote the 2-norm of a vector $x$ by $\norm{x}$,
and the Frobenius norm of a matrix $M$ by $\norm{M}_F$.
The inner product of two vectors $x,y$ is denoted by $\inprod{x}{y}$.	
	If $S$ is a symmetric positive semidefinite matrix, then we denote the weighted norm of a vector $x$ with the weight matrix $S$ by $\norm{x}_{S} :=  \sqrt{\inprod{x}{S x}}$.

\section{Generalized distance weighted discrimination}\label{sec-GeneralDWD}

This section gives details on the optimization problems underlying the  distance weighted discrimination.
Let $(x_i,y_i)$, $i=1,\ldots,n$, be the training data where
$x_i\in \R^d$ is the feature vector  and $y_i\in\{+1,-1\}$ is its corresponding class label. We let $X\in \R^{d\times n}$ be the  matrix whose columns are the $x_i$'s, and $y=[y_1,\ldots,y_n]^T$.
In linear discrimination, we attempt to separate the vectors in the two classes by
a hyperplane $H= \{ x\in \R^d\mid w^T x + \beta =0 \}$, where $w\in \R^d$ is the
unit normal and $|\beta|$ is its distance to the origin.
Given a point $z\in \R^d$, the signed distance between $z$ and the hyperplane $H$ is given by
$w^Tz + \beta$.
For binary classification where the label $y_i \in \{-1,1\}$, we want
\begin{equation*}
y_i(\beta+x_i^Tw)\geq 1-\xi_i \quad \forall\; i=1,...,n,
\end{equation*}
where we have added a slack variable $\xi \geq 0$ to allow the possibility that
the positive and negative data points may not be separated cleanly  by the hyperplane.
In matrix-vector notation, we need
\begin{equation}
r \;:=\; Z^Tw+\beta y +\xi \;\geq\; \bfone,
\label{eq-r}
\end{equation}
where $Z = X \diag{y}$ and
$\bfone\in \mathbb{R}^n$ is the vector of ones.

In SVM,  $w$ and $\beta$ are chosen by maximizing the minimum residual, i.e.,
\begin{eqnarray}
\max \Big\{ \delta - C \inprod{\bfone}{\xi} \mid Z^T w + \beta y +\xi \geq \delta \bfone, \; \xi \geq 0, \; w^Tw\leq 1 \Big\},
\label{eq-svm}
\end{eqnarray}
where $C > 0$ is a tuning parameter to control the level of penalization on $\xi$.
For the DWD approach introduced in \cite{MarronTodd-2007}, $w$ and $\beta$ are chosen instead by minimizing the sum of reciprocals of the $r_i$'s, i.e.,
\begin{eqnarray}
\min \Big\{   \sum_{i=1}^n \frac{1}{r_i} + C \inprod{\bfone}{\xi} \mid
r= Z^T w + \beta y +\xi, \; r > 0, \; \xi \geq 0, \; w^Tw\leq 1, \; w\in \R^d
\Big\}.
\label{eq-dwd}
\end{eqnarray}
Detailed discussions on the connections between the DWD model \eqref{eq-dwd} and
the SVM model \eqref{eq-svm} can  be found in \cite{MarronTodd-2007}.
The DWD optimization problem \eqref{eq-dwd} is shown to be equivalent to a second-order
cone programming problem in \cite{MarronTodd-2007} and  hence it can be
solved by interior-point methods such as those implemented in
the solver SDPT3 (\citealt{SDPT3}).

Here we design an algorithm which is capable of solving large scale generalized DWD problems
of the following form:
\begin{eqnarray}
\min \Big\{ \Phi(r,\xi) := \sum_{i=1}^n \theta_q (r_i) +C \inprod{ e}{ \xi} 
\mid \; Z^T w + \beta y +\xi -r =0, \;\;
 \norm{w} \leq 1,\;\xi \geq 0
 \Big\}
\label{eq-primal}
\end{eqnarray}
where $e\in \R^n$ is a given positive vector such that $\norm{e}_\infty =1$ (the last condition is for the purpose
of normalization).
The exponent $q$ can be any given positive number,  though the values of
most interest are likely to be $q=0.5,1,2,4$, and $\theta_q (r_i)$ is the function defined by
\begin{eqnarray*}
	\theta_q (t) = 
		\frac{1}{t^q} &\mbox{if $t>0$, } \;\mbox{and} \;\;\;
		\theta_q (t) =  \infty &\mbox{if $t\le0$. } 
\end{eqnarray*}
Observe that in addition to allowing for a general exponent $q$ in \eqref{eq-primal}, we also allow
for a nonuniform weight $e_i > 0$ in the penalty term for each $\xi_i$.
By a simple change of variables and modification of the data vector $y$, \eqref{eq-primal} can also include the
case where the terms in $\sum_{i=1}^n \frac{1}{r_i^q}$ are weighted non-uniformly.
For brevity, we omit the details.

\begin{proposition} Let  $\kappa = \frac{q+1}{q} q^{\frac{1}{q+1}}$. 
	The dual of problem \eqref{eq-primal} is given as follows:
	\begin{eqnarray}
	- \min_{\alp} \Big\{ \Psi(\alp) := \norm{Z\alp} - \kappa\sum_{i=1}^n \alp_i^{\frac{q}{q+1}}
	\mid 0 \leq \alp \leq Ce , \inprod{y}{\alp} =0
	\Big\},
	\label{eq-dual}
	\end{eqnarray}
\end{proposition}
\begin{proof}
	Consider the Lagrangian function associated with \eqref{eq-primal}:
	\begin{eqnarray*}
		&& \hspace{-0.7cm}
		L(r,w,\beta,\xi;\alpha,\eta,\lam) =  \mbox{$\sum_{i=1}^n$} \theta_q (r_i) + C \inprod{e}{\xi} -\inprod{\alp}{Z^Tw+\beta y + \xi-r}
		+ \frac{\lam}{2}(\norm{w}^2-1)  - \inprod{\eta}{\xi}
		\\
		&=& \mbox{$\sum_{i=1}^n$}  \theta_q (r_i) +
		\inprod{r}{\alp } +  \inprod{\xi} {Ce -\alp-\eta}  -\beta\inprod{y}{\alp} - \inprod{w}{Z\alp} + \frac{\lam}{2}(\inprod{w}{w}-1),
	\end{eqnarray*}
	where $r \in \R^n $, $w\in \R^d$, $\beta\in\R$, $\xi\in\R^n$, $\alp\in\R^n$,  $\lam, \eta \geq 0$.
	Now
	\begin{eqnarray*}
		&& \inf_{r_i}\Big\{ \theta_q (r_i) + \alp_i r_i \Big\}
		= \left\{ \begin{array}{ll} \kappa\, \alp_i^{\frac{q}{q+1}}
			&\mbox{if $\alp_i \geq 0$}
			\\[0pt]
			-\infty &\mbox{if $\alp_i < 0$}
		\end{array} \right.
		\\[5pt]
		&& \inf_{w} \Big\{ -\inprod{Z\alp}{w} + \frac{\lam}{2}\norm{w}^2 \Big\}
		=
		\left\{ \begin{array}{ll}
			-\frac{1}{2\lam}\norm{Z\alp}^2 & \mbox{if $\lam > 0$} \\[0pt]
			0 &\mbox{if $\lam=0$, $Z\alp = 0$} \\[0pt]
			-\infty & \mbox{if $\lam=0$, $Z\alp\not=0$}
		\end{array} \right.
		\\[5pt]
		&& \inf_{\xi} \Big\{  \inprod{\xi} {Ce -\alp-\eta} \Big\}
		= \left\{\begin{array}{ll}
			0 &\mbox{if $Ce-\alp-\eta=0$} \\[0pt]
			-\infty &\mbox{otherwise}
		\end{array} \right.,
		\\[5pt]
		&& \inf_{\beta} \Big\{ -\beta\inprod{y}{\alp}  \Big\}
		= \left\{
		\begin{array}{ll}
			0 &\mbox{if $\inprod{y}{\alp}=0$} \\[0pt]
			-\infty &\mbox{otherwise}
		\end{array} \right..
	\end{eqnarray*}
	Let $F_D = \{ \alp\in\R^n \mid 0\leq \alp\leq Ce,  \inprod{y}{\alp}=0 \}$. Hence
	\begin{eqnarray*}
		\min_{r,w,\beta,\xi} L(r,w,\beta,\xi;\alpha,\eta,\lam)
		=\left\{\begin{array}{l}
			\kappa \sum_{i=1}^n \alp_i^{\frac{q}{q+1}}
			- \frac{1}{2\lam}\norm{Z\alp}^2 - \frac{\lam}{2}, \quad\mbox{if $\lam > 0$, $\alp\in F_D$},
			\\[0pt]
			\kappa \sum_{i=1}^n \alp_i^{\frac{q}{q+1}}, \quad \mbox{if $\lam =0$, $Z\alp=0$, $\alp\in F_D$},
			\\[0pt]
			-\infty, \quad \mbox{if $\lam =0$, $Z\alp\not=0$, $\alp\in F_D$, or $\alp\not\in F_D$.}
		\end{array} \right.
	\end{eqnarray*}
	Now for $\alp \in F_D$,  we have
	\begin{eqnarray*}
		& \max_{\lam \geq 0, \eta \geq 0} \big\{ \min_{r,w,\beta,\xi} L(r,w,\beta,\xi;\alpha,\eta,\lam)\big \}
		= \kappa \sum_{i=1}^n \alp_i^{\frac{q}{q+1}} -\norm{Z\alp}.
	\end{eqnarray*}
	From here, we get the required dual problem.
\end{proof}

\bigskip
It is straightforward to show that the feasible regions of \eqref{eq-primal} and \eqref{eq-dual}
both have nonempty interiors. Thus optimal solutions for both problems exist and they
satisfy
the following KKT (Karush-Kuhn-Tucker) optimality conditions:
\begin{eqnarray}
\begin{array}{ll}
Z^T w + \beta y + \xi - r = 0, & \inprod{y}{\alp} = 0, \\[0pt]
r > 0,\; \alp > 0, \quad \alp \leq C e, \quad \xi \geq 0,  & \inprod{Ce - \alp}{\xi} = 0, \\[0pt]
\alp_i = \frac{q}{r_i^{q+1}}, \; i=1,\ldots,n,  &
\mbox{either} \; w = \frac{Z\alp}{\norm{Z\alp}}, \;\; \mbox{or}\;\; Z\alp = 0, \; \norm{w}^2 \leq 1.
\end{array} \label{eq-optimality}
\end{eqnarray}

Let $(r^*,\xi^*,w^*,\beta^*)$ and $\alp^*$ be an optimal solution of
\eqref{eq-primal} and \eqref{eq-dual}, respectively.
Next we analyse some properties of the optimal solution. In particular,
we show that the optimal solution $\alp^*$
is bounded away from 0.
\begin{proposition} There exists a  positive $\delta$ such that
	$
	\alp^*_i \geq \delta \;\; \forall\; i=1,\ldots,n.
	$
\end{proposition}
\begin{proof} For convenience, let $F_P = \{ (r,\xi,w,\beta) \mid
	Z^T w +\beta y + \xi -r = 0, \norm{w} \leq 1, \xi \geq 0\}$ be the
	feasible region of \eqref{eq-primal}. Since $(\bfone,\bfone,0,0)\in F_P$, we have that
	$$
	C e_{\min} \xi^*_i \leq
	C \inprod{e}{\xi^*} \leq \Phi(r^*,\xi^*,w^*,\beta^*) \leq \Phi(\bfone,\bfone,0,0)
	= n + C \mbox{$\sum_{i=1}^n$} e_i
	\quad \forall\; i=1,\ldots,n,
	$$
	where $e_{\min} = \min_{1\leq i \leq n} \{ e_i\}$.
	Hence we have $0\leq \xi^* \leq  \varrho \bfone$, where 
	$\varrho := \frac{n + C \sum_{i=1}^{n} e_i}{C e_{\min}}$.
	
	Next, we establish a bound for $|\beta^*|$. Suppose $\beta^* > 0$.
	Consider an index $i$ such that $y_i = -1$. Then
	$
	0 < \beta^* = Z_i^T w^* + \xi^*_i-r^*_i \leq \norm{Z_i}\norm{w^*} + \xi^*_i
	\leq K + \varrho,
	$
	where $Z_i$ denotes the $i$th column of $Z$,
	$K  = \max_{1\leq j \leq n}\{\norm{Z_j}\}$. On the other hand, if
	$\beta^* < 0$, then we consider an index $k$ such that $y_k = 1$, and
	$
	0< -\beta^* = Z_k^T w^* + \xi^*_k-r^*_k \leq K + \varrho.
	$
	To summarize, we have that $|\beta^*| \leq K + \varrho$.
	
	Now we can establish an upper bound for $r^*$. For any $i=1,\ldots,n$, we have that
	$$
	r^*_i  = Z_i^T w^* + \beta^* y_i + \xi^*_i  \leq \norm{Z_i}\norm{w^*}
	+|\beta^*| + \xi_i^* \leq 2(K+\varrho).
	$$
	From here, we get
	$
	\alp^*_i = \frac{q}{ (r_i^*)^{q+1}}  \geq  \delta := \frac{q}{ (2K+2\varrho)^{q+1}} \quad\forall\;
	i=1,\ldots,n.
	$
	This completes the proof of the proposition.
\end{proof}

\section{An inexact SGS-based ADMM for large scale DWD problems}\label{sec-sGSADMM}

We can rewrite the model \eqref{eq-primal} as:
\begin{eqnarray*}
	\min \Big\{
	\sum_{i=1}^n
	\theta_q (r_i)+ C \inprod{e}{\xi} +\delta_B(w) + \delta_{\R^n_+}(\xi)
	\mid \; Z^T w +\beta y+\xi-r=0, \; \; w\in \R^d, \;  r,\xi\in \R^n
	\Big\}
\end{eqnarray*}
where $B=\{w \in \R^d \mid \norm{w}\leq 1\}$. 
Here, both $\delta_B(w)$ and $\delta_{\R^n_+}(\xi)$ are infinity indicator functions. In general, an infinity indicator function 
over a set $\mathcal{C}$  is defined by:
\begin{eqnarray*}
	\delta_\mathcal{C}(x):=
\begin{cases}
	0, &\quad \text{if } x \in \mathcal{C}; \\
	+\infty, &\quad \text{otherwise.}
\end{cases}
\end{eqnarray*}	

The model above is a convex
minimization problem with three nonlinear blocks.
By introducing an auxiliary variable $u=w$, we can reformulate it as:
\begin{eqnarray}
\begin{array}{ll}
\min &\sum_{i=1}^n \theta_q (r_i)+ C \inprod{e}{\xi}+\delta_B(u) + \delta_{\R^n_+}(\xi)
\\[8pt]
\text{s.t.} &Z^T w+\beta y+\xi-r=0\\[3pt]
&  D(w- u) =0, \;\;  w,u\in\R^d, \; \beta\in \R,\; r,\xi\in \R^n,
\end{array}
\label{eq-gen}
\end{eqnarray}
where $D\in \R^{d\times d}$ is a given positive scalar multiple of the identity
matrix which is introduced
for the purpose of scaling the variables.

For a given parameter $\sig > 0$,
the augmented Lagrangian function associated with \eqref{eq-gen} is given by
\begin{eqnarray*}
	\begin{array}{rcl}
		L_{\sigma}(r,w,\beta,\xi,u;\alpha,\rho) &=&\sum_{i=1}^n \theta_q (r_i)
		+ C \inprod{e}{\xi}+\delta_B(u) + \delta_{\R^n_+}(\xi)
		+\frac{\sigma}{2}\norm{Z^T w+\beta y +\xi -r -\sigma^{-1}\alpha}^2
		\\[8pt]
		&& +\frac{\sigma}{2}\norm{D(w-u)-\sigma^{-1}\rho}^2
		-\frac{1}{2\sigma}\norm{ \alpha}^2 -\frac{1}{2 \sigma}\norm{\rho}^2.
	\end{array}
\end{eqnarray*}
The algorithm which we will design later is based on recent progress
in algorithms for solving multi-block convex conic programming.
In particular, our algorithm is designed based on
the inexact ADMM algorithm in \cite{CST} and we made essential use of
the inexact symmetric Gauss-Seidel
decomposition theorem in \cite{LST} to solve the subproblems
arising in each iteration of the algorithm.

We can view \eqref{eq-gen} as a linearly constrained nonsmooth
convex
programming problem with three blocks of variables grouped as
$(w,\beta)$, $r$, $(u,\xi)$.
The template for our inexact
sGS based ADMM is described next.
Note that the
subproblems need not be solved exactly as long as they satisfy some prescribed accuracy.

\begin{description}
	\item[Algorithm 1.] {\bf An inexact sGS-ADMM for solving \eqref{eq-gen}.}
	\\[5pt]
	Let $\{ \varepsilon_k\}$ be a  summable sequence of nonnegative
	nonincreasing numbers.
	Given an initial iterate $(r^0,w^0,\beta^0,\xi^0,u^0)$ in the feasible region of \eqref{eq-gen}, and
	$(\alpha^0,\rho^0)$ in the dual feasible region of \eqref{eq-gen}, choose  a $d\times d$ symmetric positive semidefinite matrix $\cT$, and
	perform the
	following steps in each iteration.
	\item[Step 1a.]
	Compute
	$$
	(\bar{w}^{k+1},\bar{\beta}^{k+1})
	\approx  \mbox{argmin}_{w,\beta}\; \Big\{L_\sigma (r^{k},w,\beta,\xi^k,u^k;\alpha^k,\rho^k)
	+ \frac{\sig}{2}\norm{w-w^k}_{\cT}^2 \Big\}.
	$$

	In particular, $(\bar{w}^{k+1},\bar{\beta}^{k+1})$ is an approximate solution to the
	following $(d+1)\times (d+1)$ linear system of equations:
	\begin{eqnarray}
	\underbrace{\begin{bmatrix} ZZ^T+ D^2 +\cT & Zy\\[5pt] (Zy)^T& y^Ty
		\end{bmatrix} }_{A}
	\begin{bmatrix} w \\[5pt] \beta
	\end{bmatrix}
	=  \bar{h}^k :=
	\begin{bmatrix}
	-Z(\xi^k-r^k-\sigma^{-1}\alpha^k)+D^2 u^k + D(\sigma^{-1}\rho^k) + \cT w^k\\[5pt]
	-y^T(\xi^k-r^k-\sigma^{-1}\alpha^k)
	\end{bmatrix}. \quad
	\label{eq-linsys}
	\end{eqnarray}
	We require the residual of
	the approximate solution $(\bar{w}^{k+1},\bar{\beta}^{k+1})$
	to satisfy
	\begin{eqnarray}
	\norm{\bar{h}^k - A[\bar{w}^{k+1};\bar{\beta}^{k+1}]} \leq \varepsilon_k.
	\label{eq-linsys-tol}
	\end{eqnarray}
	
	\item[Step 1b.] Compute $r^{k+1}\approx\mbox{argmin}_{r\in \R^n}\; L_{\sigma}(r,\bar{w}^{k+1},\bar{\beta}^{k+1},\xi^k,u^k;\alpha^k,\rho^k)$. Specifically, by observing that the
	objective function in this subproblem is actually separable in $r_i$ for $i=1,\ldots,n$,
	we can compute $r^{k+1}_i$ as follows:
	\begin{eqnarray}
	\begin{array}{lll}
	r^{k+1}_i &\approx& \arg \min_{r_i} \Big\{ \theta_q (r_i) +\frac{\sigma}{2} \norm{r_i-c^k_i}^2 \Big\}\\
	&=& \arg \min_{r_i > 0} \Big\{  \frac{1}{r_i^q}
	+\frac{\sigma}{2} \norm{r_i-c^k_i}^2 \Big\} \quad \forall \;i=1,\ldots,n,
	\end{array} \label{eq-1c}
	\end{eqnarray}
	where $c^k=Z^T \bar{w}^{k+1}+y\bar{\beta}^{k+1}+\xi^k-\sigma^{-1}\alpha^k$.
	The details on how the above one-dimensional problems are solved
	will be given later.
	The solution $r_i^{k+1}$ is deemed to be sufficiently accurate if
	\begin{eqnarray*}
		\Big|-\frac{q}{ (r_i^{k+1})^{q+1}} + \sig (r^{k+1}_i- c^k_i) \Big|\leq \varepsilon_k/\sqrt{n}
		\quad \forall\; i=1,\ldots,n.
	\end{eqnarray*}
	\item[Step 1c.] Compute
	$$
	(w^{k+1},\beta^{k+1}) \approx \mbox{argmin}_{w,\beta}\; \Big\{
	L_{\sigma}(r^{k+1},w,\beta,\xi^k,u^k;\alpha^k,\rho^k) + \frac{\sig}{2}\norm{w-w^k}_\cT^2
	\Big\},$$
	which amounts to solving the
	linear system of equations \eqref{eq-linsys}
	but with $r^k$ in the right-hand side vector $\bar{h}^k$ replaced by $r^{k+1}$.
	Let $h^k$ be the new right-hand side vector. We require the approximate solution
	to satisfy the accuracy condition that
	$$
	\norm{h^k - A[w^{k+1};\beta^{k+1}]}\leq 5 \varepsilon_k.
	$$
	Observe that the accuracy requirement here is more relaxed than that
	stated in \eqref{eq-linsys-tol} of Step 1a. The reason for doing so is that
	one may hope to use the solution $(\bar{w}^{k+1},\bar{\beta}^{k+1})$
	computed in Step 1a as an approximate solution for the current subproblem.
	If   $(\bar{w}^{k+1},\bar{\beta}^{k+1})$ indeed satisfies the above
	accuracy condition, then one can simply
	set $(w^{k+1},\beta^{k+1}) = (\bar{w}^{k+1},\bar{\beta}^{k+1})$ and
	the cost of solving this new subproblem can be saved.
	
	\item[Step 2.]
	Compute $(u^{k+1},\xi^{k+1})= \mbox{argmin}_{u,\xi}\; L_{\sigma}(r^{k+1},w^{k+1},\beta^{k+1},\xi,u;\alpha^k,\rho^k)$. By observing that the objective function is
	actually separable in $u$ and $\xi$, we can compute $u^{k+1}$ and $\xi^{k+1}$ separately
	as follows:
	\begin{eqnarray*}
		u^{k+1}&=& \arg \min \left\{ \delta_B (u) +\frac{\sigma}{2}\norm{D(u-g^k)}^2 \right\} \;=\; \begin{cases} g^k & \mbox{if  $\norm{g^k} \leq 1$}
			\\[0pt]
			g^k/\norm{g^k} & \mbox{otherwise}
		\end{cases},
		\\
		\xi^{k+1}&=&\Pi _{\R_{+}^n}\Big( r^{k+1}  -Z^T w^{k+1} -y \beta^{k+1}
		+\sigma^{-1}\alpha^k -\sigma^{-1}C e\Big),
	\end{eqnarray*}
	where $g^k=w^{k+1}-\sigma^{-1}D^{-1}\rho^k$, and $\Pi_{\R^n_+}(\cdot)$ denotes
	the projection onto $\R^n_+$.
	\item[Step 3.] Compute
	\begin{eqnarray*}
		\alpha^{k+1}&=& \alpha^k -\tau \sigma (Z^T w^{k+1}+y\beta^{k+1}+\xi^{k+1}-r^{k+1}),\\
		\rho^{k+1}&=& \rho^k -\tau \sigma D (w^{k+1}-u^{k+1}),
	\end{eqnarray*}
	where $\tau \in (0,(1+\sqrt{5})/2)$ is the steplength which is typically
	chosen to be $1.618$.
\end{description}

In our implementation of Algorithm 1, we choose the summable sequence $\{ \varepsilon_k\}_{k\geq 0}$ to
be $\varepsilon_k = c/(k+1)^{1.5}$ where $c$ is a constant that is inversely
proportional to $\norm{Z}_F$.
Next we discuss the computational cost of Algorithm 1.
As we shall see later, the most computationally intensive steps in each iteration
of the above algorithm are in solving the linear systems of equations of the form
\eqref{eq-linsys} in Step 1a and 1c. The detailed analysis of their computational costs
will be presented in  subsection \ref{subsec-linsys}.
All the other steps can be done
in at most $O(n)$ or $O(d)$ arithmetic operations, together with the
computation of $Z^T w^{k+1}$, which costs $2dn$ operations if we do not
take advantage of any possible sparsity in $Z$.

\subsection{Convergence results}

We have the following convergence theorem for the inexact sGS-ADMM, established by Chen, Sun and Toh in \citet[Theorem 1]{CST}. This theorem guarantees the convergence of our algorithm to optimality, as a merit over the possibly non-convergent directly extended semi-proximal ADMM.

\begin{theorem}
	Suppose that the system \eqref{eq-optimality} has at least one solution. Let $\{(r^k,w^k,\beta^k,\xi^k,u^k;\alpha^k,\rho^k)\}$ be the sequence generated by the inexact sGS-ADMM in Algorithm 1. Then the sequence $\{(r^k,w^k,\beta^k,\xi^k,u^k)\}$ converges to an optimal solution of problem \eqref{eq-gen} and the sequence $\{(\alpha^k,\rho^k)\}$ converges to an optimal solution to the dual of problem \eqref{eq-gen}.
\end{theorem}
\begin{proof} 	
	In order to apply the convergence result in \cite{CST}, we need to express \eqref{eq-gen}
	as follows:
	\begin{eqnarray}
	\min \Big\{p(r) + f(r,w,\beta)  + q(\xi,u)  + g(\xi,u)
	\mid\;
	 A_1^* r  + A_2^* [w;\beta] + B^*[\xi; u]  = 0 \Big\}	
	\end{eqnarray}
	where $p(r) = \sum_{i=1}^n \theta_q (r_i),
		\;\; f(r,w,\beta) \equiv 0,
		\quad q(\xi,u) = \delta_B(u)  + C \inprod{e}{\xi} + \delta_{\R^n_+}(\xi),  \;\; g(\xi,u)  \equiv 0,$
	\begin{eqnarray*}
		& A_1^* = \left(\begin{array}{c} -I \\[0pt] 0 \end{array} \right),\;
		A_2^* = \left(\begin{array}{cc} Z^T & y \\[0pt] D & 0 \end{array} \right),\;
		B^* = \left(\begin{array}{cc} I & 0 \\[0pt] 0 & -D \end{array} \right).&
	\end{eqnarray*}
	Next  we need to consider the following matrices:
	\begin{eqnarray*}
		\left(\begin{array}{c}
			A_1 \\[0pt] A_2
		\end{array}\right) \Big( A_1^*,\; A_2^* \Big) + \left(\begin{array}{ccc}
		0 & 0 & 0 \\[0pt] 0 & \cT & 0 \\[0pt] 0 & 0 & 0
	\end{array}\right)
	=\left( \begin{array}{ccc}
		I &  [-Z^T, -y] \\[0pt]
		[-Z^T,-y]^T & M
	\end{array}\right), \quad BB^* = \left(\begin{array}{cc}
	I & 0 \\[0pt] 0 & D^2
\end{array}\right),
\end{eqnarray*}
where
$$
M = \left( \begin{array}{cc}
ZZ^T + D^2 + \cT & Zy \\[0pt]
(Zy)^T & y^T y
\end{array}\right) \succ 0.
$$
One can show that $M$ is positive definite by using the Schur complement lemma.
With the conditions that $M\succ 0$ and $BB^*\succ 0$, the conditions
in Proposition 4.2 of \cite{CST} are satisfied, and hence the convergence of Algorithm 1 follows by using Theorem 1 in \cite{CST}.
\end{proof}

We note here that the convergence analysis  in \cite{CST} is highly nontrivial. 
But it is motivated by the
	proof for the simpler case of an
	exact semi-proximal ADMM that is available in Appendix B of the paper by
	\cite{FPST}.
	In that paper, one can see that the convergence proof is based on the descent 
	property of 
	a certain function, while the augmented Lagrangian function itself does not have
	such a descent property.

\subsection{Numerical computation of the subproblem \eqref{eq-1c} in Step 1b}

In the presentation of Algorithm 1, we have described how the subproblem in each step can be solved
except for the subproblem \eqref{eq-1c} in Step 1b. Now we discuss how it can be solved. Observe that
for each $i$, we need to solve a one-dimensional problem of the form:
\begin{eqnarray}
\min \Big\{ \varphi(s):= \frac{1}{s^q} + \frac{\sig}{2} (s-a)^2 \mid s > 0
\Big\},
\label{eq-varphi}
\end{eqnarray}
where $a$ is given. It is easy to
see that $\varphi(\cdot)$ is a convex function and it has a unique minimizer
in the domain $(0,\infty)$.
The optimality condition for \eqref{eq-varphi} is given by
$$
s- a = \frac{q\sig^{-1}}{s^{q+1}},
$$
where the unique minimizer $s^*$
is determined by the intersection of the line $s\mapsto s-a$ and the
curve $s\mapsto \frac{q \sig^{-1}}{s^{q+1}}$ for $s> 0$.
We propose to use Newton's method to find the
minimizer, and the template is given as follows. Given an initial iterate $s^0$, perform the
following iterations:
$$
s_{k+1} = s_k - \varphi'(s_k)/\varphi''(s_k) = s_k \left( \
\frac{q(q+2)\sig^{-1} + a s^{q+1}_k}{q(q+1)\sig^{-1} +  s^{q+2}_k}\right), \quad k=0,1,\ldots
$$
Since $\varphi^{\prime\prime}(s^*)  > 0$, Newton's method
would have a local quadratic convergence rate, and we would expect
it to converge in a small number of iterations, say less than $20$,
if a good initial point $s^0$ is given. In solving the subproblem
\eqref{eq-1c} in Step 1b, we always use the previous solution $r_i^k$ as the
initial point to warm-start Newton's method.
If a good initial point is not available, one can use the bisection technique to find one. In our tests, this technique was however never used.

Observe that the computational cost for solving the subproblem \eqref{eq-1c} in Step 1b is
$O(n)$ if Newton's method converges within a fixed number of iterations (say 20)
for all $i=1,\ldots,n$. Indeed, in our experiments, the average number of
Newton iterations required to solve \eqref{eq-varphi} for each of the instances is less than $10$.

\subsection{Efficient techniques to solve the linear system 
	\eqref{eq-linsys}}
\label{subsec-linsys}

Observe that in each iteration of Algorithm 1, we need to solve a $(d+1)\times (d+1)$ linear system of
equations \eqref{eq-linsys}
with the same coefficient matrix $A$. For large scale problems where $n$ and/or $d$ are large,
this step would constitute the most expensive part of the algorithm.
In order to solve such a linear system efficiently,
we design different techniques to solve it, depending on the dimensions $n$ and $d$.
We consider the following cases.

\subsubsection{The case where $d\ll n$ and $d$ is moderate}\label{subsec-directSolver}

This is the most straightforward case where
we set $\cT =0$, and
we solve
\eqref{eq-linsys} by computing the Cholesky factorization of the coefficient matrix $A$.
The cost of computing $A$ is
$2 n d^2$ arithmetic operations. Assuming that $A$ is stored, then we can
compute its Cholesky factorization at the cost of $O(d^3)$ operations, which
needs only to be performed once at the very beginning of Algorithm 1.
After that, whenever we need to solve the linear system
\eqref{eq-linsys}, we compute the right-hand-side vector at
the cost of $2nd$ operations and
solve two $(d+1)\times (d+1)$
triangular systems of linear equations at the cost of $2d^2$ operations.

\subsubsection{The case where $n\ll d$ and $n$ is moderate}\label{subsec-smw}
\def\hD{\widehat{D}}

In this case, we also set $\cT=0$. But solving the large $(d+1)\times (d+1)$
system of linear equations \eqref{eq-linsys} requires more thought.
In order to avoid inverting the high dimensional matrix $A$ directly,
we make use of the Sherman-Morrison-Woodbury formula to
get $A^{-1}$ by inverting a much smaller
$(n+1)\times (n+1)$ matrix
as shown in the following proposition.

\begin{proposition} The coefficient matrix $A$ can be rewritten as follows:
	\begin{eqnarray}
	A &=& \hD + U E U^T, \quad
	U = \left[\begin{array}{cc}
	Z  & 0\\ y^T  &\norm{y}
	\end{array}\right], \quad E = \diag{I_n,-1},
	\label{eq-Anew}
	\end{eqnarray}
	where $\hD = \diag{D,\norm{y}^2}$.
	It holds that
	\begin{eqnarray}
	A^{-1} &=& \hD^{-1} - \hD^{-1} U H^{-1} U^T \hD^{-1},
	\label{eq-Ainv}
	\end{eqnarray}
	where
	\begin{eqnarray}
	H = E^{-1} + U^T \hD^{-1} U =
	\left[\begin{array}{cc}
	I_n + Z^T D^{-1} Z +  yy^T/\norm{y}^2  &  y/\norm{y}\\[3pt] y^T/\norm{y}  & 0
	\end{array}\right].
	\label{eq-H}
	\end{eqnarray}
\end{proposition}
\begin{proof} It is easy to verify that \eqref{eq-Anew} holds and we omit the details.
	To get \eqref{eq-Ainv}, we only need to apply the Sherman-Morrison-Woodbury formula in \citet[p.50]{GVbook}
	to \eqref{eq-Anew} and perform some simplifications.
\end{proof}

\medskip

Note that in making use of \eqref{eq-Ainv} to compute
$A^{-1}\bar{h}^k$, we need to find $H^{-1}$. A rather cost effective way to
do so is to express $H$ as follows and use the Sherman-Morrison-Woodbury formula to find
its inverse:
\begin{eqnarray*}
	H =J+ \bar{y}\bar{y}^T, \quad
	J = diag(I_n + Z^TD^{-1} Z,-1), \quad \bar{y} = [y/\norm{y}; \; 1].
\end{eqnarray*}
With the above expression for $H$, we have that
\begin{eqnarray*}
	H^{-1} = J^{-1} - \frac{1}{1+\bar{y}^T J^{-1}\bar{y}} (J^{-1}\bar{y} ) (J^{-1}\bar{y})^T.
\end{eqnarray*}
Thus
to solve \eqref{eq-linsys}, we first compute the $n\times n$ matrix $I_n + Z^TD^{-1} Z$ in \eqref{eq-H} at
the cost of $2dn^2$ operations. Then we compute its Cholesky factorization  at the cost of $O(n^3)$ operations.
(Observe that even though we are solving a $(d+1)\times (d+1)$ linear system of equations
for which $d\gg n$, we only need to compute the Cholesky factorization of
a much smaller $n\times n$ matrix.)
Also, we need to compute $J^{-1}\bar{y}$ at the cost of
$O(n^2)$ operations by using the previously computed Cholesky factorization.
These computations only need to be performed
once at the beginning of Algorithm 1. After that, whenever we need to
solve a linear system of the form \eqref{eq-linsys},
we can compute $\bar{h}^k$ at the cost of $2nd$ operations, and then
make use of $\eqref{eq-Ainv}$ to get $A^{-1}\bar{h}^k$ by solving two $n\times n$
triangular systems of linear equations at the cost of $2n^2$ operations, and
performing two matrix-vector multiplications involving $Z$ and $Z^T$ at a
total cost of $4nd$ operations. To summarize, given the Cholesky factorization of the first diagonal block of $H$, the cost of solving \eqref{eq-linsys} via \eqref{eq-Ainv} is $6nd + 2n^2$ operations.

\subsubsection{The case where $d$ and $n$ are both large}

The purpose of introducing the proximal term $\frac{1}{2}\norm{w-w^k}_\cT^2$
in Steps 1a  and 1c is to make the computation of the solutions of the subproblems
easier.
However, one should note that adding the proximal term typically will make the algorithm converge more slowly, and the  deterioration will become worse for
larger $\norm{\cT}$. Thus in practice, one would need to strike a balance between choosing
a symmetric positive semidefinite matrix $\cT$ to make the computation easier while not slowing down
the algorithm by too much.

In our implementation, we first attempt to
solve the subproblem in Step 1a (similarly for 1c) without adding a proximal term by setting
$\cT=0$. In particular, we solve the linear system \eqref{eq-linsys}
by using
a preconditioned symmetric quasi-minimal residual (PSQMR)  iterative solver (\citealt{PSQMR})
when both $n$ and $d$ are large.
	Basically, it is a variant of the Krylov subspace method similar to the idea in GMRES (\citealt{SaadBook}). For more details on the PSQMR algorithm, the reader is referred to the appendix.	
In each step of the PSQMR solver, the main cost is in performing the matrix-vector
multiplication with the coefficient matrix $A$, which costs
$4nd$ arithmetic operations.
As the number of steps taken by an iterative solver to solve \eqref{eq-linsys}
to the required accuracy \eqref{eq-linsys-tol} is dependent on the
conditioning of   $A$, in the event that the
solver requires more than $50$ steps to solve \eqref{eq-linsys},
we would switch to
adding a suitable non-zero proximal term $\cT$ to make the subproblem in Step 1a easier to solve. 	 	

The most common and natural choice of $\cT$ to make the subproblem in Step 1a easy to solve is
to set $\cT = \lambda_{\max}I- ZZ^T$, where
$\lam_{\max}$ denotes the largest eigenvalue of $ZZ^T$. In this case
the corresponding linear system \eqref{eq-linsys} is very easy to solve. More precisely, for the
linear system in \eqref{eq-linsys}, we can first compute $\bar{\beta}^{k+1}$ via the
Schur complement equation in a single variable followed by computing $\bar{w}^k$ as follows:
\begin{eqnarray}
\begin{array}{l}
\big(y^Ty - (Zy)^T (\lam_{\max} I + D)^{-1} (Zy) \big)\beta = \bar{h}^k_{d+1}
- (Zy)^T (\lam_{\max} I + D)^{-1} \bar{h}^k_{1:d},
\\[5pt]
\bar{w}^{k+1} = (\lam_{\max}I + D)^{-1} (\bar{h}^k_{1:d} - (Zy)\bar{\beta}^{k+1}),
\end{array}
\label{eq-schur}
\end{eqnarray}
where $\bar{h}^k_{1:d}$ denotes the vector extracted from the first $d$ components
of $\bar{h}^k$.
In our implementation, we pick a $\cT$ which is less conservative
than the above natural choice as follows.
Suppose we have computed the first $\ell$ largest eigenvalues of $ZZ^T$ such that $\lam_1\geq\ldots\geq\lam_{\ell-1} > \lam_\ell$,  and their
corresponding orthonormal set of eigenvectors, $v_1,\ldots,v_\ell$.
We pick $\cT$ to be
\begin{eqnarray}
\cT = \lam_\ell I + \mbox{$\sum_{i=1}^{\ell-1}$} (\lam_i -\lam_\ell) v_i v_i^T - ZZ^T,
\label{eq-T}
\end{eqnarray}
which can be proved to be positive semidefinite by using the spectral decomposition
of $ZZ^T$. In practice, one would typically pick $\ell$ to be a small integer, say $10$,
and compute the first $\ell$ largest eigenvalues and their corresponding eigenvectors
via variants of the Lanczos method. 
The most expensive step in each iteration of the Lanczos method is a matrix-vector multiplication, which 
requires $O(d^2)$ operations. In general, the cost of computing the first few largest eigenvalues of $ZZ^T$ 
is much cheaper than that of computing the full eigenvalue decomposition.
In {\sc Matlab}, such a computation can be
done by using the routine {\tt eigs}.
To solve \eqref{eq-linsys}, we need the inverse of $ZZ^T + D+\cT$. Fortunately,
when $D = \mu I_d$, it can easily be inverted with
$$
(ZZ^T + D + \cT)^{-1} = (\mu+\lam_\ell)^{-1} I_d
+ \mbox{$\sum_{i=1}^{\ell-1}$}
\big( (\mu+\lam_i)^{-1}-(\mu+\lam_\ell)^{-1}\big) v_iv_i^T.
$$
One  can  then compute $\bar{\beta}^k$ and $\bar{w}^k$ as
in \eqref{eq-schur} with $(\lam_{\max} I + D)^{-1}$ replaced by the above
inverse.


\section{Experiments}\label{sec-experiments}
In this section, we test the performance of our inexact sGS-ADMM method on several publicly available data sets. The numerical results presented in the subsequent subsections are obtained from a computer with processor specifications: Intel(R) Xeon(R) CPU E5-2670 @ 2.5GHz (2 processors) and 64GB of RAM, running on a 64-bit Windows Operating System.

\subsection{Tuning the penalty parameter}

In the DWD model \eqref{eq-gen}, we see that it is important to make a suitable choice of the penalty parameter $C$.
In \cite{MarronTodd-2007}, it has been noticed that a reasonable choice for the penalty parameter when the exponent $q=1$ is a large constant divided by the square of a typical distance between the $x_i$'s, where the typical distance, $dist$, is defined as the median of the pairwise Euclidean distances between classes. We found out that in a more general case, $C$ should be inversely proportional to $dist^{q+1}$. On the other hand, we observed that a good choice of $C$ also depends on the sample size $n$ and the dimension of features $d$. In our numerical experiments, we empirically set the value of $C$ to be
$10^{q+1}\max \big\{1,\frac{10^{q-1} \log(n)\max\{1000,d\}^{\frac{1}{3}}}{dist^{q+1}}\big\}$,
where $\log(\cdot)$ is the natural logarithm.

\subsection{Scaling of data}

A technique which is very important in implementing ADMM based methods in practice to
achieve fast convergence
is the data scaling technique.
Empirically, we have observed that
it is good to scale the matrix $Z$ in \eqref{eq-gen} so that the magnitude of all the blocks in the equality constraint would be roughly the same. Here we choose the scaling factor to be $Z_{\rm scale} = \sqrt{\norm{X}_F}$, 
where $\norm{\cdot}_F$ is the Frobenius norm. Hence the optimization model in \eqref{eq-gen} becomes:
\begin{eqnarray}
\begin{array}{ll}
\min &\sum_{i=1}^n \frac{1}{r_i^q}+ C \inprod{e}{\xi}+\delta_{\widetilde{B}}(\tilde{u}) + \delta_{\R^n_+}(\xi)
\\[8pt]
\text{s.t.} & \widetilde{Z}^T \,\tilde{w}
+\beta y+\xi-r=0,\; r > 0,\\[0pt]
&  D(\tilde{w}- \tilde{u}) =0, \;\;  \tilde{w}, \tilde{u}\in\R^d, \; r,\xi\in \R^n,
\end{array}
\label{eq-gen-scale}
\end{eqnarray}
where $\widetilde{Z} = \frac{Z}{Z_{\rm scale}}$, 
$\tilde{w} = Z_{\rm scale} w$, $\tilde{u} = Z_{\rm scale} u$, and
$\widetilde{B}=\{\tilde{w} \in \R^d \mid \norm{\tilde{w}}\leq Z_{\rm scale}\}$. Therefore, if we have computed an optimal solution $(r^*,\tilde{w}^*,\beta^*,\xi^*,\tilde{u}^*)$ of
\eqref{eq-gen-scale}, then
$(r^*,Z^{-1}_{\rm scale}\tilde{w}^*,\beta^*,\xi^*,Z^{-1}_{\rm scale}\tilde{u}^*)$
would be an optimal solution of \eqref{eq-gen}.

\subsection{Stopping condition for inexact sGS-ADMM}\label{subsec-stopping}

We measure the accuracy of an approximate optimal solution $(r,w,\beta,\xi,u,\alpha,\rho)$ for \eqref{eq-gen-scale} based on the KKT optimality conditions \eqref{eq-optimality} by defining the following relative residuals:
\begin{eqnarray*}
	\arraycolsep=1.4pt\def\arraystretch{1.5}
	\begin{array}{lllll}
		\eta_{C_1} = \frac{|y^T \alp|}{1+C}, & \eta_{C_2} = \frac{|\xi^T (Ce-\alp)|}{1+C}, & \eta_{C_3} = \frac{\norm{\alp - s}^2}{1+C} \text{ with } s_i = \frac{q}{r_i^{q+1}} , \\
		\eta_{P_1} = \frac{\norm{\widetilde{Z}^T \tilde{w} + \beta y + \xi - r}}{1+C}, 
		& \eta_{P_2} = \frac{\norm{D(\tilde{w}-\tilde{u})}}{1+C}, & \eta_{P_3} = \frac{\max\{\norm{\tilde{w}}-Z_{\rm scale},0\}}{1+C} , 
		\\
		\eta_{D_1} = \frac{\norm{\min\{0,\alp\}}}{1+C}, & \eta_{D_2} = \frac{\norm{\max\{0,\alp-Ce\}}}{1+C} ,
	\end{array}
\end{eqnarray*}
where $Z_{\rm scale}$ is a scaling factor which has been discussed in the last subsection.
Additionally, we calculate the relative duality gap by:
$$
\eta_{gap} := \frac{|\primobj-\dualobj|}{1+|\primobj|+|\dualobj|},
$$
where
$\primobj = \sum_{i=1}^n \frac{1}{r_i^q}+ C \inprod{e}{\xi}, \; 
\dualobj =  \kappa\sum_{i=1}^n \alp_i^{\frac{q}{q+1}}- Z_{\rm scale}\norm{\widetilde{Z}\alp}, \text{ with } 
\kappa = \frac{q+1}{q} q^{\frac{1}{q+1}}.$
We should emphasize that although for machine learning problems, a high accuracy solution is
usually not required,  it is important however to use
the KKT optimality conditions as the stopping criterion to find a moderately
accurate solution in order to design a robust solver.

We terminate the solver when $\max\{\eta_P,\eta_D\}<10^{-5}$,  $\min\{\eta_C,\eta_{gap}\}<\sqrt{10^{-5}}$, and $\max\{\eta_C,\eta_{gap}\} < 0.05$.
Here, $\eta_C = \max\{\eta_{C_1},\eta_{C_2},\eta_{C_3}\}, \; \eta_P = \max\{\eta_{P_1},\eta_{P_2},\eta_{P_3}\}$, and $ \eta_D = \max\{\eta_{D_1},\eta_{D_2}\}$.
Furthermore, the maximum number of iterations is set to be 2000.

\subsection{Adjustment of Lagrangian parameter $\sigma$ }

Based upon some preliminary experiments, we set our initial Lagrangian parameter $\sig$ to be $\sigma_0=\min\{10C, n\}^q$, where $q$ is the exponent in \eqref{eq-gen},
and adapt the following strategy to update $\sigma$ to improve the convergence speed of the algorithm in practice:

\begin{description}
	\item[Step 1.] Set $\chi=\frac{\eta_P}{\eta_D}$, where $\eta_P$ and $\eta_D$ are defined in subsection \ref{subsec-stopping};
	\item[Step 2.] If $\chi>\theta$, set $\sig_{k+1} = \zeta \sig_k$; elseif $\frac{1}{\chi}>\theta$, set $\sig_{k+1} = \frac{1}{\zeta} \sig_k$.
\end{description}
Here we empirically set $\theta$ to be 5 and $\zeta$ to be 1.1. Nevertheless, if we have either $\eta_P\ll\eta_D$ or $\eta_D\ll\eta_P$, then we would increase $\zeta$ accordingly, say 2.2 if $\max\{\chi,\frac{1}{\chi}\}>500$ or 1.65 if $\max\{\chi,\frac{1}{\chi}\}>50$.

\subsection{Performance of the sGS-ADMM on UCI data sets}
In this subsection, we test our algorithm on instances from the UCI data repository  (\citealt{UCI}). 
The datasets we have chosen here are all classification problems with two classes.
However, the size for each class may not be balanced. To tackle the case of uneven class proportions, we use the weighted DWD model discussed in \cite{qiao}. Specifically,
we consider the  model \eqref{eq-primal} using $e = \bfone$ and the term
$\sum_{i=1}^n 1/r_i^{q}$ is replaced by $\sum_{i=1}^n \tau_i^q/ r_i^{q}$,  with
the weights $\tau_i$  given as follows:
\begin{eqnarray*}
	\tau_i = \left\{ \begin{array}{ll}
		\frac{\tau_-}{\max\{\tau_+,\tau_-\}} & \mbox{if $y_i=+1$} \\[5pt]
		\frac{\tau_+}{\max\{\tau_+,\tau_-\}} & \mbox{if $y_i=-1$}
	\end{array}, \right.
\end{eqnarray*}
where $\tau_\pm =
\big(|n_\pm| K^{-1}\big)^{\frac{1}{1+q}}$.
Here $n_{\pm}$ is the number of data points with class label $\pm 1$ respectively and $K:={n}/{\log(n)}$ is a normalizing factor.

\begin{center}
	\spacingset{1.2}
	\begin{footnotesize}	
		\begin{longtable}{| c | c | c | c | c | c| c| c| c|}
			\hline
			\multicolumn{1}{|c}{Data} & \multicolumn{1}{|c|}{$n$}
			& \multicolumn{1}{|c|}{$d$}
			& \multicolumn{1}{|c|}{$C$}
			& \multicolumn{1}{|c|}{Iter}  & \multicolumn{1}{|c|}{Time (s)}
			& \multicolumn{1}{|c|}{psqmr$|$double}
			& \multicolumn{1}{|c|}{Train-error (\%)}
			\\ \hline
			
			\endhead
			 a8a &22696 & 123 &6.27e+02 &201  &2.28  &   0$|$201 &15.10
			\\[3pt]\hline
			a9a &32561 & 123 &6.49e+02 &201  &2.31  &   0$|$201 &14.93
			\\[3pt]\hline
			covtype &581012 &  54 &3.13e+03 &643  &104.34  &   0$|$191 &23.74
			\\[3pt]\hline
			gisette &6000 &4972 &1.15e+04 &101  &39.69  &   0$|$ 49 &0.17
			\\[3pt]\hline
			gisette-scale &6000 &4956 &1.00e+02 &201  &59.73  &   0$|$201 &0.00
			\\[3pt]\hline
			ijcnn1 &35000 &  22 &4.23e+03 &401  &3.16  &   0$|$401 &7.77
			\\[3pt]\hline
			mushrooms &8124 & 112 &3.75e+02 & 81  &1.09  &   0$|$ 81 &0.00
			\\[3pt]\hline
			real-sim &72309 &20958 &1.55e+04 &210  &47.69  & 875$|$210 &1.45
			\\[3pt]\hline
			w7a &24692 & 300 &5.95e+02 &701  &4.84  &   0$|$701 &1.17
			\\[3pt]\hline
			w8a &49749 & 300 &6.36e+02 &906  &9.43  &   0$|$906 &1.20
			\\[3pt]\hline
			rcv1 &20242 &44505 &9.18e+03 & 81  &9.18  & 234$|$ 49 &0.63
			\\[3pt]\hline
			leu &  38 &7129 &1.00e+02 &489  &2.84  &   0$|$489 &0.00
			\\[3pt]\hline
			prostate & 102 &6033 &1.00e+02 & 81  &3.19  &   0$|$ 81 &0.00
			\\[3pt]\hline
			farm-ads &4143 &54877 &3.50e+03 & 81  &6.92  & 792$|$ 81 &0.14
			\\[3pt]\hline
			dorothea & 800 &88120 &1.00e+02 & 51  &4.44  &   0$|$ 51 &0.00
			\\[3pt]\hline
			url-svm &256000 &685896 &1.18e+06 &121  &294.55  & 364$|$121 &0.01
			\\[3pt]\hline

			\caption{The performance of our inexact sGS-ADMM method on the UCI data sets.  }
			\label{table-UCI}
		\end{longtable}
	\end{footnotesize}		
\end{center}

Table \ref{table-UCI} presents the number of iterations and runtime required, as well as training error produced when we perform our inexact sGS-ADMM algorithm to solve {16} data sets. Here, the running time is the total time spent in reading the training data and in solving
the DWD model. The timing
for getting the best penalty parameter C is excluded.
The results are generated using the exponent $q=1$.
In the table, ``psqmr" is the iteration count for the preconditioned symmetric quasi-minimal residual method for solving the linear system \eqref{eq-linsys}. A `0' for ``psqmr" means that we are using a direct solver as mentioned in subsection \ref{subsec-directSolver} and \ref{subsec-smw}. {Under the column ``double" in Table \ref{table-UCI}, we also record the number of iterations  for which the extra Step 1c is executed to ensure the convergence of Algorithm 1.

Denote the index set $S=\{i \mid y_i [\text{sgn}(\beta+x_i^T w)]\le 0, i=1,\ldots,n\}$ for which the data instances are categorized wrongly, where $sgn(x)$ is the sign function. The training and testing errors are both defined by $\frac{|S|}{n}\times 100 \%$, where $|S|$ is the cardinality of the set $S$.}

Our algorithm is capable of solving all the data sets, even when the size of the data matrix is huge.
In addition, for data with unbalanced class size, such as w7a and w8a, our algorithm is able to produce a classifier with small training error.

\subsection{Comparison with other solvers}

In this subsection, we compare our inexact sGS-ADMM method for solving \eqref{eq-primal}
via  \eqref{eq-gen}
with
the primal-dual interior-point method implemented in \cite{SDPT3} and used in \cite{MarronTodd-2007}. We also compare our method with
the directly extended (semi-proximal) ADMM (using
the aggressive step-length 1.618)  even though the latter's convergence is not
guaranteed.
Note that the directly extended ADMM we have implemented here follows exactly the
same design used for sGS-ADMM, except that
we switch off the additional Step 1c in the Algorithm 1.
We should emphasize that our directly extended ADMM is not a simple
adaption of the classical ADMM, but instead incorporates all the sophisticated techniques
we have developed for sGS-ADMM.

We will report our computational results for two different values of the exponent, $q=1$ and $q=2$, in Tables \ref{table-comparison} and  \ref{table-comparison-2}, respectively.

\begin{center}
	\spacingset{1.2}
	\begin{footnotesize}
		\footnotesize
		\setlength{\tabcolsep}{1.5pt}
		\begin{longtable}{|| c | c | c | c || c | c | c|| c| c |c || c| c| c||}
			\hline
			\multicolumn{4}{||c||}{exponent $q=1$} & \multicolumn{3}{c||}{sGS-ADMM} & \multicolumn{3}{c||}{directADMM} & \multicolumn{3}{c||}{IPM} \\
			\hline
			\multicolumn{1}{||c|}{Data} & \multicolumn{1}{|c|}{$n$} & \multicolumn{1}{|c|}{$d$}
			& \multicolumn{1}{|c||}{$C$}  & \multicolumn{1}{|c|}{Time (s)} & \multicolumn{1}{|c|}{Iter}
			& \multicolumn{1}{|c||}{Error (\%)} & \multicolumn{1}{|c|}{Time (s)} & \multicolumn{1}{|c|}{Iter}
			& \multicolumn{1}{|c||}{Error (\%)} & \multicolumn{1}{|c|}{Time (s)} & \multicolumn{1}{|c|}{Iter}
			& \multicolumn{1}{|c||}{Error (\%)}
			\\ \hline
			
			\endhead
			
			 a8a &22696 & 123 &6.27e+02&2.28 &201 &15.10&2.01 &219 &15.10&1321.20 & 47 &15.09
			\\[3pt]\hline
			a9a &32561 & 123 &6.49e+02&2.31 &201 &14.93&2.12 &258 &14.93&2992.60 & 43 &14.93
			\\[3pt]\hline
			covtype &581012 &  54 &3.13e+03&104.34 &643 &23.74&100.26 &700 &23.74 & - & - & - 
			\\[3pt]\hline
			gisette &6000 &4972 &1.15e+04&39.69 &101 &0.17&33.14 & 70 &0.25&2403.40 & 62 &20.50*
			\\[3pt]\hline
			gisette-scale &6000 &4956 &1.00e+02&59.73 &201 &0.00&152.05 &2000 &0.00&49800.58 & 84 &17.80*
			\\[3pt]\hline
			ijcnn1 &35000 &  22 &4.23e+03&3.16 &401 &7.77&2.95 &501 &7.77&1679.34 & 34 &7.77
			\\[3pt]\hline
			mushrooms &8124 & 112 &3.75e+02&1.09 & 81 &0.00&1.42 &320 &0.00&136.67 & 50 &0.00
			\\[3pt]\hline
			real-sim &72309 &20958 &1.55e+04&47.69 &210 &1.45&89.55 &702 &1.42 & - & - & - 
			\\[3pt]\hline
			w7a &24692 & 300 &5.95e+02&4.84 &701 &1.17&8.52 &2000 &1.16&4474.13 & 53 &1.17
			\\[3pt]\hline
			w8a &49749 & 300 &6.36e+02&9.43 &906 &1.20&13.58 &2000 &1.16 & - & - & - 
			\\[3pt]\hline
			rcv1 &20242 &44505 &9.18e+03&9.18 & 81 &0.63&16.07 &245 &0.63&9366.41 & 43 &0.17
			\\[3pt]\hline
			leu &  38 &7129 &1.00e+02&2.84 &489 &0.00&5.35 &2000 &0.00&1.33 & 11 &0.00
			\\[3pt]\hline
			prostate & 102 &6033 &1.00e+02&3.19 & 81 &0.00&4.13 &2000 &0.00&7.73 & 19 &0.00
			\\[3pt]\hline
			farm-ads &4143 &54877 &3.50e+03&6.92 & 81 &0.14&4.20 & 61 &0.14&642.79 & 54 &0.24
			\\[3pt]\hline
			dorothea & 800 &88120 &1.00e+02&4.44 & 51 &0.00&37.31 &2000 &0.00&15.53 & 23 &0.00
			\\[3pt]\hline
			url-svm &256000 &685896 &1.18e+06&294.55 &121 &0.01&264.52 &195 &0.00 & - & - & - 
			\\[3pt]\hline

			\caption{Comparison between the performance of our inexact sGS-ADMM, directly extended ADMM ``directADMM", and the interior point method ``IPM" on the UCI data sets. A `*' next to the error in the table means that the problem set cannot be solved properly by the respective solver; `-' means the algorithm cannot solve the dataset due to insufficient computer memory.}
			\label{table-comparison}
		\end{longtable}
	\end{footnotesize}
\end{center}

Table \ref{table-comparison} reports the runtime, number of iterations required as well as the training error of 3 different solvers for solving the UCI data sets. We can observe that the interior point method is almost always the slowest to achieve optimality compared to the other two solvers, despite requiring the least number of iterations, especially when the sample size $n$ is large.
The inefficiency of the interior-point method is caused by its need to solve an $n\times n$ linear system of equations in each iteration, which could be very expensive if $n$ is large.
In addition, it cannot solve the DWD problem where $n$ is huge due to the excessive computer memory needed to store the large $n\times n$ matrix.

On the other hand, our inexact sGS-ADMM method outperforms the directly extended (semi-proximal) ADMM for 
9 out of {16 cases} in terms of runtime. For the other cases, we are only slower by a relatively small margin. Furthermore, when our algorithm outperforms the directly extended ADMM, it often shortens the runtime by a large margin.
In terms of number of iterations, for {14 out of 16 cases}, the directly extended ADMM requires at least the same number of iterations as our inexact sGS-ADMM method. We can say that our algorithm is remarkably efficient and it further possesses a
convergence guarantee. In contrast,  the directly extended ADMM
is not guaranteed to converge although it is also very efficient when it
does converge. We can observe that the directly extended ADMM sometimes
would take many more iterations to solve a problem compared to
our inexact sGS-ADMM, especially for the instances in Table \ref{table-comparison-2},
possibly because the lack of a convergence guarantee makes it difficult for the method to find a sufficiently accurate approximate optimal solution.

To summarize, our inexact sGS-ADMM method is an efficient yet convergent algorithm for solving the primal form of the DWD model. It is also able to solve large scale problems which cannot be handled by the interior point method.

\begin{center}
	\spacingset{1.2}
	\begin{footnotesize}
		\footnotesize
		\setlength{\tabcolsep}{1.5pt}
		\begin{longtable}{|| c | c | c | c || c | c | c|| c| c |c || c| c| c||}
			\hline
			\multicolumn{4}{||c||}{exponent $q=2$} & \multicolumn{3}{c||}{sGS-ADMM} & \multicolumn{3}{c||}{directADMM} & \multicolumn{3}{c||}{IPM} \\
			\hline
			\multicolumn{1}{||c|}{Data} & \multicolumn{1}{|c|}{$n$} & \multicolumn{1}{|c|}{$d$}
			& \multicolumn{1}{|c||}{$C$}  & \multicolumn{1}{|c|}{Time (s)} & \multicolumn{1}{|c|}{Iter}
			& \multicolumn{1}{|c||}{Error (\%)} & \multicolumn{1}{|c|}{Time (s)} & \multicolumn{1}{|c|}{Iter}
			& \multicolumn{1}{|c||}{Error (\%)} & \multicolumn{1}{|c|}{Time (s)} & \multicolumn{1}{|c|}{Iter}
			& \multicolumn{1}{|c||}{Error (\%)}
			\\ \hline
			
			\endhead
			
			 a8a &22696 & 123 &1.57e+04&2.77 &248 &15.10&2.97 &533 &15.11&7364.34 & 45 &15.17
			\\[3pt]\hline
			a9a &32561 & 123 &1.62e+04&2.92 &279 &14.94&5.78 &1197 &14.94&16402.05 & 48 &14.95
			\\[3pt]\hline
			covtype &581012 &  54 &1.52e+05&81.63 &368 &23.74&275.29 &2000 &23.74 & - & - & - 
			\\[3pt]\hline
			gisette &6000 &4972 &1.01e+06&39.72 &101 &0.03&28.42 & 81 &0.03&2193.74 & 55 &0.00
			\\[3pt]\hline
			gisette-scale &6000 &4956 &1.00e+03&88.89 &439 &0.00&147.80 &2000 &1.20&31827.33 & 53 &0.00
			\\[3pt]\hline
			ijcnn1 &35000 &  22 &2.69e+05&2.76 &233 &7.94&4.80 &1056 &7.87&1959.68 & 38 &7.98
			\\[3pt]\hline
			mushrooms &8124 & 112 &7.66e+03&1.59 &301 &0.00&3.63 &2000 &0.17&594.33 & 52 &0.00
			\\[3pt]\hline
			real-sim &72309 &20958 &1.10e+06&43.20 &180 &1.51&23.56 &181 &1.51 & - & - & - 
			\\[3pt]\hline
			w7a &24692 & 300 &1.44e+04&3.96 &473 &1.15&7.83 &2000 &2.69&5448.78 & 48 &1.18
			\\[3pt]\hline
			w8a &49749 & 300 &1.54e+04&7.07 &543 &1.13&13.57 &2000 &2.68 & - & - & - 
			\\[3pt]\hline
			rcv1 &20242 &44505 &4.69e+05&9.47 & 81 &1.16&7.54 &101 &1.16&8562.76 & 47 &0.24*
			\\[3pt]\hline
			leu &  38 &7129 &1.00e+03&1.72 &204 &0.00&4.52 &2000 &0.00&2.82 & 23 &0.00
			\\[3pt]\hline
			prostate & 102 &6033 &1.00e+03&3.32 &111 &0.00&3.87 &2000 &0.00&8.33 & 21 &0.00
			\\[3pt]\hline
			farm-ads &4143 &54877 &5.21e+04&37.79 &401 &0.19&8.31 &146 &0.19&343.13 & 43 &0.27
			\\[3pt]\hline
			dorothea & 800 &88120 &1.00e+03&4.39 & 51 &0.00&31.49 &2000 &0.00&16.84 & 25 &0.00
			\\[3pt]\hline
			url-svm &256000 &685896 &5.49e+07&452.81 &205 &0.02&587.46 &490 &0.02 & - & - & - 
			\\[3pt]\hline

			\caption{
			Same as Table \ref{table-comparison} but 
			 for $q=2$.}
			\label{table-comparison-2}
		\end{longtable}
	\end{footnotesize}
\end{center}

Table \ref{table-comparison-2} reports the runtime, number of iterations required
as well as the training error of 3 different solvers for solving the UCI data sets for the case when $q=2$. Again, we can see that the interior point method is almost always the slowest to converge to optimality.

Our sGS-ADMM algorithm outperforms the directly extended ADMM algorithm in {12 out of 16} data sets in terms of runtime. 
In terms of the number of iterations, it has the best performance among almost all the data sets. On the other hand, for  {8} data sets, the number of iterations required by the directly extended ADMM hits the maximum iterations allowed, probably implying nonconvergence of the method. For the  interior point method, it
takes an even longer time to solve the problems
compared to the case when $q=1$. This is due to an increase in the number of constraints generated
in the second-order cone programming formulation of the DWD model with $q=2$.

The numerical result we obtained in this case is consistent with the one we obtained for the case $q=1$. This further shows the merit of our algorithm in a more general setting. We could also expect the similar result when the exponent is 4 or 8.

\subsection{Comparison with LIBSVM and LIBLINEAR}

In this subsection, we will compare the performance of our DWD model to the state-of-the-art model support 
vector machine (SVM). We apply our sGS-ADMM algorithm as the DWD model and use
LIBSVM in \cite{LIBSVM} as well as LIBLINEAR in \cite{LIBLINEAR} to implement the SVM model. LIBSVM is a 
general solver for solving SVM models with different kernels; while LIBLINEAR is a solver highly specialized 
in solving SVM with linear kernels. LIBLINEAR is a fast linear classifier; in particular, we would apply it to the 
dual of L2-regularized L1-loss support vector classification problem. We would like to emphasize that the solution given by LIBSVM using linear kernel and that given by LIBLINEAR is not exactly the same. This may be
due to the reason that LIBLINEAR has internally preprocessed the data and 
assumes that there is no bias term in the model.

{The parameters used in LIBSVM are chosen to be the same as in \cite{unifapg}, whereas for LIBLINEAR, we make use of the default parameter $C = 1$ for all the datasets.}

Table \ref{table-comparison-LIBSVM} shows the runtime and number of iterations needed for 
solving the binary classification problem via the DWD and SVM models respectively on the UCI datasets. It also gives the training and testing classification error produced by the three algorithms.  Note that 
the training time excludes the time for computing the best penalty parameter $C$.
The stopping tolerance for all algorithms is set to be $10^{-5}$. 
For LIBLINEAR, we observed that the default maximum number of iteration is breached for many datasets. 
Thus we increase the maximum number of iteration from the default 1000 to 20000.

In terms of runtime, LIBLINEAR is almost always the fastest to solve the problem, except for the
{two largest datasets (rcv1,url-svm) for which our algorithm is about 2-3 times faster}. 
Note  that the maximum iteration is reached for {these two datasets}. On the other hand, 
LIBSVM is almost always the slowest solver. It may only be faster than sGS-ADMM for small datasests (3 cases). 
Our algorithm can be 50-100 times faster than LIBSVM when solving large data instances. Furthermore,
LIBSVM may have the risk of not being able to handle extremely large-scaled datasets. For example, 
it cannot solve the biggest dataset (url-svm) within 24 hours. 

In terms of training and testing error, we may observe from the table that the DWD and SVM models 
produced comparable training classification errors, although there are some discrepancies due to the 
differences in the models and penalty parameters used. On the other hand, the testing errors vary across 
different solvers. For most datasets, the DWD model (solved by sGS-ADMM) produced smaller (sometimes much smaller) testing errors than
the other algorithms (8 cases); whereas the SVM model (solved by LIBLINEAR) produced the 
worst testing errors among all algorithms {(5 cases)}. The discrepancy between the testing errors 
given by LIBSVM and LIBLINEAR may be due to the different treatment of the bias term in-built into the algorithms.

It is reasonable to claim that our algorithm is more efficient than the extremely fast solver LIBLINEAR
in solving large data instances even though our algorithm is designed for the more complex
DWD model compared to the simpler SVM model.
Moreover, our algorithm for solving the DWD model is able to produce testing errors
which are generally better than those produced by LIBLINEAR for solving the SVM model.

\begin{center}
	\spacingset{1.2}
	\begin{footnotesize}
		\footnotesize
		\setlength{\tabcolsep}{0.8pt}
		\begin{longtable}{|| c | c | c |c || c | c | c|c|| c|c | c|c || c| c| c|c||}
			\hline
			\multicolumn{4}{||c||}{} & \multicolumn{4}{c||}{DWD via sGS-ADMM} & \multicolumn{4}{c||}{SVM via LIBLINEAR} & \multicolumn{4}{c||}{SVM via LIBSVM} \\
			\hline
			\multicolumn{1}{||c|}{Data} & \multicolumn{1}{|c|}{$n$} & \multicolumn{1}{|c|}{$d$} & \multicolumn{1}{|c||}{$n_{test}$}
			& \multicolumn{1}{|c|}{Time} & \multicolumn{1}{|c|}{Iter} 	& \multicolumn{1}{|c|}{$Err_{tr}$} & \multicolumn{1}{|c||}{$Err_{test}$} & \multicolumn{1}{|c|}{Time} & \multicolumn{1}{|c|}{Iter} 
			& \multicolumn{1}{|c|}{$Err_{tr}$} & \multicolumn{1}{|c||}{$Err_{test}$} & \multicolumn{1}{|c|}{Time} & \multicolumn{1}{|c|}{Iter} & \multicolumn{1}{|c|}{$Err_{tr}$} 	& \multicolumn{1}{|c||}{$Err_{test}$}
			\\ \hline			
			\endhead
			
			 a8a &22696 & 123  &9865 &2.28 &201 &15.10  &14.67 &0.69 &20000 &15.32  &14.87&50.45 &34811 &15.44  &14.80
			\\[3pt]\hline
			a9a &32561 & 123  &16281 &2.31 &201 &14.93  &15.19 &0.79 &6475 &15.01  &15.02&93.91 &25721 &15.24  &15.03
			\\[3pt]\hline
			covtype &581012 &  54 & / &104.34 &643 &23.74 & / &23.53 &20000 &23.69 & /&19641.19 &224517 &23.70 & /
			\\[3pt]\hline
			gisette &6000 &4972  &1000 &39.69 &101 &0.17  &3.00 &4.34 &132 &0.00  &10.70&85.93 &14818 &0.40  &14.30
			\\[3pt]\hline
			gisette-scale &6000 &4956  &1000 &59.73 &201 &0.00  &2.50 &33.07 &484 &0.00  &2.30&152.27 &15738 &0.40  &2.20
			\\[3pt]\hline
			ijcnn1 &35000 &  22  &91701 &3.16 &401 &7.77  &7.82 &0.55 &11891 &8.70  &8.35&27.04 &10137 &9.21  &8.76
			\\[3pt]\hline
			mushrooms &8124 & 112 & / &1.09 & 81 &0.00 & / &0.19 &326 &0.00 & /&0.64 &1003 &0.00 & /
			\\[3pt]\hline
			real-sim &72309 &20958 & / &47.69 &210 &1.45 & / &7.56 &1190 &1.07 & /&3846.07 &52591 &1.37 & /
			\\[3pt]\hline
			w7a &24692 & 300  &25057 &4.84 &701 &1.17  &1.30 &0.43 &1689 &1.28  &10.12&14.79 &62830 &1.37  &1.38
			\\[3pt]\hline
			w8a &49749 & 300  &14951 &9.43 &906 &1.20  &1.31 &2.45 &13095 &1.18  &9.64&63.27 &124373 &1.36  &1.43
			\\[3pt]\hline
			rcv1 &20242 &44505  &677399 &9.18 & 81 &0.63  &5.12 &33.98 &20000 &0.15  &5.17&819.32 &33517 &0.62  &13.82
			\\[3pt]\hline
			leu &  38 &7129  &  34 &2.84 &489 &0.00  &5.88 &0.30 & 27 &0.00  &20.59&0.59 &184 &0.00  &17.65
			\\[3pt]\hline
			prostate & 102 &6033 & / &3.19 & 81 &0.00 & / &2.73 &353 &0.00 & /&2.54 &1063 &0.00 & /
			\\[3pt]\hline
			farm-ads &4143 &54877 & / &6.92 & 81 &0.14 & / &6.11 &5364 &0.14 & /&10.34 &15273 &0.14 & /
			\\[3pt]\hline
			dorothea & 800 &88120  & 350 &4.44 & 51 &0.00  &5.14 &0.83 & 11 &0.00  &8.86&7.48 &4553 &0.00  &7.71
			\\[3pt]\hline
			url-svm &256000 &685896 & / &294.55 &121 &0.01 & / &660.39 &20000 &0.01 & /& - & - & -  & -
			\\[3pt]\hline
					
			\caption{Comparison between the performance of our inexact sGS-ADMM on DWD model with LIBLINEAR and LIBSVM on SVM model. $n_{test}$ is the size of testing sample, $Err_{tr}$ is the percentage of training error; while $Err_{test}$ is that of the testing error. `-' means the result cannot be obtained within 24 hours, and `/' means test sets are not available.}
			\label{table-comparison-LIBSVM}
		\end{longtable}
	\end{footnotesize}
\end{center}

\section{Conclusion}\label{sec-conclusion}

In this paper, by making use of the recent advances in ADMM from the work in \cite{STY}, \cite{LST} and \cite{CST}, we proposed a convergent 3-block inexact symmetric
Gauss-Seidel-based semi-proximal ADMM algorithm for solving large scale DWD problems. We applied the algorithm successfully to the primal formulation of the DWD model and designed highly efficient routines to solve the subproblems arising in each of the inexact sGS-ADMM iterations. Numerical experiments for the cases when the exponent equals to 1 and 2 demonstrated that our algorithm is capable of solving large scale problems, even when the sample size and/or the feature dimension is huge. In addition, it is also highly efficient while guaranteeing the convergence to optimality. As a conclusion, we have designed an efficient method for solving the binary classification model through DWD.

\bigskip
\begin{center}
{\large\bf SUPPLEMENTARY MATERIAL}
\end{center}

\begin{description}

\item[Matlab package:] MATLAB-package containing codes to perform the sGS-ADMM algorithm for the DWD model described in the article. The package also contains datasets used for experiment purpose in the article. It could be downloaded from http://www.math.nus.edu.sg/$\sim$mattohkc/DWDLarge.zip.

\item[R package:] R-package containing codes to perform the sGS-ADMM algorithm for the DWD model described in the article. It would be available at CRAN.

\bigskip
\begin{center}
	{\large\bf APPENDIX}
\end{center}
\item[PSQMR algorithm:] In the following we shall present a brief version of the PSQMR method to solve the linear system $Ax = b$. The following is a simplified algorithm without preconditioner.

Given initial iterate $x_0$, define each of the following:
$r_0 = b - Ax_0; \; q_0 = r_0; \; \tau_0 = \norm{q_0}; \; \theta_0 = 0; \; d_0 = 0$. 

Then for each iteration $k$ until convergent condition is met, compute:
\begin{eqnarray*}
r_k &=& r_{k-1} - \frac{r_{k-1}^T r_{k-1}}{q_{k-1}^T Aq_k} Aq_k \\
\theta_k &=& \frac{\norm{r_k}}{\tau_{k-1}} \\
c_k &=& \frac{1}{\sqrt{1+\theta_k^2}} \\
\tau_k &=& \tau_{k-1} \theta_k c_k \\
d_k &=& c_k^2 \theta_{k-1}^2 d_{k-1} + c_k^2 \frac{r_{k-1}^T r_{k-1}}{q_{k-1}^T Aq_k} Aq_k \\
x_k &=& x_{k-1} + d_k \\
q_k &=& r_k + \frac{r_{k}^T r_{k}}{r_{k-1}^T r_{k-1}} q_{k-1} \\	
\end{eqnarray*}
The final $x_k$ obtained is an approximated solution to the linear system.

\end{description}

\spacingset{0.9}
\bibliography{Bibliography-DWD}
\end{document}